\documentclass{article}
\usepackage{graphicx, pdfsync}
\usepackage[colorlinks=true, citecolor=red]{hyperref}
\usepackage{amsmath,amssymb,amsfonts,MnSymbol, mathrsfs}
\usepackage{floatflt,wrapfig,caption}

\usepackage{algorithm}
\usepackage{xcolor}
\usepackage{listings}
\usepackage[colorlinks=true, citecolor=red]{hyperref}

\usepackage{tikz,pgfplots,bm}
\usetikzlibrary{patterns}

\newtheorem{theorem}{Theorem}
\newtheorem{proposition}[theorem]{Proposition}%
\newtheorem{example}{Example}%
\newtheorem{remark}{Remark}%

\newtheorem{lemma}{Lemma}
\newtheorem{hypothesis}{Hypothesis}
\newtheorem{corollary}{Corollary}

\def\e{{\rm e}}
\def\R{{\mathbb{R}}}

\newcommand{\ds}{\displaystyle}
\newcommand{\be}{\begin{equation}}
\newcommand{\ee}{\end{equation}}

\def\p{\partial}

\def\eeq#1{\begin{eqnarray}#1\end{eqnarray}}
\def\eeqn#1{\begin{eqnarray*}#1\end{eqnarray*}}

\def\vx{{\bf x}}
\def\vy{{\bf y}}

\def\vn{{\bf n}}

\def\vom{{\bm\omega}}
\def\SS{{\mathbb{S}^2}}

\def\eps{{\epsilon}}

\def\dd{\mathrm{d}}

\title{{\Huge Reflective Conditions for  Radiative Transfer  in Integral Form with H-Matrices}}
\author{Olivier Pironneau\footnote{ LJLL, Sorbonne Universit\'e and Acad\'emie des Sciences}~  and Pierre-Henri Tournier\footnote{ LJLL, Sorbonne Universit\'e}}

\begin{document}
\maketitle
\begin{abstract}
In a recent article the authors showed that the radiative Transfer equations with multiple frequencies and scattering can be formulated as a nonlinear  integral system.  In the present article, the formulation is extended to handle reflective boundary conditions. The fixed point method to solve the system is shown to be monotone. The discretization is done with a $P^1$ Finite Element Method. The convolution integrals are precomputed at every vertices of the mesh and stored in  compressed hierarchical matrices,  using Partially Pivoted Adaptive Cross-Approximation. Then the fixed point iterations involve only matrix vector products. The method is $O(N\sqrt[3]{N}\ln N)$, with respect to the number of vertices, when everything is smooth. A numerical implementation is proposed and tested on two examples. As there are some analogies with ray tracing the programming is complex.
\end{abstract}

\nocite{*}


\section*{Introduction}
The Radiative Transport Equations (RTE)  describe the behavior of electromagnetic radiation in a domain $\Omega$ as it interacts with matter \cite{POM}. It is used to model a wide range of physical phenomena, including the propagation of light through plasma, tomography \cite{tarvainen}, atmospheric media \cite{MIH}, etc.

The RTE is derived from the basic principles of quantum and statistical mechanics; it is a partial differential equation (PDE) that describes the distribution of radiation intensity in space, time and frequencies, coupled with a budget balance equation (BBE) for the electronic temperature. The PDE takes into account both absorption and scattering of radiation by matter, as well as emission of radiation by sources, which, in the present case, will be restricted to the boundaries of the emitting material.

In \cite{DIA},\cite{FGOP3}  the authors have shown that the PDE can be converted into an integral equation for the total radiation at each point in the domain and that the coupling with the BBE can be handled by fixed point iterations. The method leads also to a general proof of existence, uniqueness and regularity of the solution. The difference with earlier studies such as \cite{fan} is in the coupling with the equation for the temperature, the BBE, or even the PDE for the temperature when diffusion is important.

In \cite{JCP} the authors have presented an implementation of the method using {\cal H}-Matrix  compression, a crucial ingredient which makes the evaluation of the integrals $O(N\sqrt[3]{N}\ln N)$ with respect to the number of vertices $N$ in the 3D mesh which discretizes the domain $\Omega$; $N\ln N$ is the complexity of the {\cal H}-Matrix approximations but each element of the matrix requires an integral along a line in the domain. Compared with a brute force solution of the equations as in \cite{JOL}, the integral method keeps a manageable computing time for problems with frequency dependent parameters.  However, it did not handle reflective boundary conditions \cite{siewert}.

\smallskip

{\cal H}-Matrix  compression \cite{hack}, \cite{beb},\cite{borm}, is a mathematical technique used to efficiently represent and manipulate large matrices that arise in a variety of applications. The technique uses a hierarchical structured representation of the matrices allowing fast and accurate numerical computations when the integrals have a convolution type integrand which decays with the distance.

{\cal H}-Matrix compression works by first defining a hierarchical geometric partitioning of the matrix into smaller and smaller submatrices. This so-called hierarchical \textit{block tree} is then traversed recursively and \textit{far-field} interaction blocks which verify  a \textit{geometric admissibility condition}\cite{beb} are compressed  by  using a low rank approximation. The resulting  {\cal H}-Matrix allows for efficient matrix-vector multiplications, among other operations of linear algebra. The technique is particularly important and popular for computational electromagnetics in integral form such as boundary element methods.

With the {\it Partially Pivoted Adaptive Cross-Approximation} (ACA) \cite{borm} only the needed  coefficients of the matrices are computed  ($r$ rows and $r$ columns, where $r$ is the rank of the approximation). However the theory requires geometrical smoothness \cite{beben}.

\smallskip

We have extended the implementation done in \cite{JCP} using \texttt{FreeFEM} \cite{FF} and \texttt{htool}\footnote{https://github.com/htool-ddm/htool}; \texttt{htool} is a parallel C++ toolbox implementing {\cal H}-Matrices, used in particular for the boundary element method in electromagnetism.
\texttt{FreeFEM} is a popular open-source software package for solving PDE systems by the finite element method (FEM). 

\texttt{FreeFEM} provides a wide range of pre-built FEM, as well as tools for mesh generation. It  has a dedicated high level programming language that allows users to meet their specific needs. \texttt{FreeFEM} also supports parallel computing with \texttt{mpi}.

One of the main advantages of \texttt{FreeFEM} for the present study is its ability to handle complex geometries and boundary conditions, especially thanks to its powerful automatic interpolation from volume to surface meshes.

Adding reflective conditions (RC) to the \texttt{FreeFEM} code presented in \cite{JCP} turned out to require solving the following difficulties:
\begin{itemize}
\item Integrate the RC into the integral formulation of the problem
\item Show that the fixed point iterations are still monotone.
\item Find a formulation compatible with the use of H-Matrices
\item Implement the method in the \texttt{FreeFEM} language.
\end{itemize}
This paper presents the solutions found to overcome these four difficulties. It ends with a numerical test proposed in \cite{koba}.

\section{The Radiative Transfer Equations}
The problem is formulated in a  domain $\Omega\subset\R^3$ with boundary $\Gamma$. The unit sphere in $\R^3$ is called $\SS$. One must find the radiation (called light from now on) intensity $I_\nu(\vx,\vom)$ at all points  $\vx\in\Omega$, for all directions all $\vom\in\SS$ and all frequencies $\nu\in\R_+$, satisfying:
\begin{align}
&\vom\cdot\nabla I_\nu+{\kappa_\nu} I_\nu={\kappa_\nu}(1-a_\nu)B_\nu(T)+{\kappa_\nu}a_\nu J_\nu\,,\quad J_\nu:=\tfrac1{4\pi}\int_{\SS}I_\nu\dd\vom\,,
\label{RTa}
\\
&\int_0^\infty{\kappa_\nu}(1-a_\nu)(J_\nu-B_\nu(T))\dd\nu=0\,,
\label{RTb}
 \\ &
I_\nu(\vx,\vom) = R_\nu(\vx,\vom) I_\nu(\vx,\vom-2(\vn\cdot\vom)\vn) + Q_\nu(\vx,\vom),
\cr & \hskip 3cm \hbox{ on }\Sigma:=\{(\vx,\vom)\in\Gamma\times\SS~:~\vom\cdot\vn(\vx)<0\,\},
\label{RTc}
\end{align}
where $B_\nu(T)=\frac{\nu^3}{\e^{\frac\nu T}-1}$ is the (rescaled) Planck function.  In the RC \eqref{RTc}, $R_\nu$ is the portion of light which is reflected and $Q_\nu$ is the light source; $\vn(\vx)$ is the outer normal of $\Gamma$ at $\vx$. $\kappa_\nu>0$ and $a_\nu\in[0,1]$ are the absorption and scattering coefficients; in general they depend on $\nu$ and $\vx$.

\begin{example}
If an object $\cal O$ inside a box $\cal B$ radiates because it is at temperature $T_0$, then, $\Omega={\cal B}\backslash {\cal O}$, $Q_\nu=Q^0 [\vom\cdot\vn]_- B_\nu(T_0)$ on ${\cal O}$ and zero elsewhere  and 
$\Sigma\subset\partial {\cal B}\times \SS$. {\rm As usual $[f]_-=-\min(f,0)$.}
\end{example}

\subsection{An Integral Formulation}
For clarity we drop the subscript $\nu$ on $\kappa,~ a$ and $I$. Assume that $\Omega$ is bounded and convex (see remark \ref{rem1}). Let
\begin{equation}\label{eqS}
S_\nu(\vx)=\kappa(1-a)B_\nu(T)+\kappa a J_\nu,
\end{equation}
For a given $\vx$ and $\vom$, let $\tau_{\vx,\vom}$ be such that $(\vx_\Sigma(\vx,\vom):=\vx-\tau_{\vx,\vom}\vom,\vom)\in\Sigma$; the method of characteristics tells us that
\begin{equation}\label{charact}
\begin{aligned}
I(\vx,\vom) &= I(\vx_\Sigma(\vx,\vom),\vom)\e^{-\int_0^{\tau_{\vx,\vom}}\kappa(\vx-\vom s)\dd s}
+
\int_0^{\tau_{\vx,\vom}} \e^{-\int_0^s\kappa(\vx-\vom s')\dd s'} S_\nu(\vx-\vom s)\dd s.
\end{aligned}
\end{equation}
Notice that $\tau_{\vx,\vom}=|\vx_\Sigma-\vx|$ (see Figure \ref{sketch}), 
therefore, let
\begin{equation}\label{gen1}
\begin{aligned}
J_\nu(\vx)&:=\tfrac1{4\pi}\int_\SS I(\vx,\vom)\dd\omega
=S^E_\nu(\vx) + {\mathcal J}[S_\nu](\vx)\quad \text{ with}
\cr
S^E_\nu(\vx)&:=\tfrac1{4\pi}\int_\SS I(\vx_\Sigma(\vx,\vom),\vom)\e^{-\int_0^{\tau_{\vx,\vom}}\kappa(\vx-\vom s)\dd s}\dd\omega, 
\cr
{\mathcal J}[S](\vx)&:=
\tfrac1{4\pi}\int_\SS\int_0^{\tau_{\vx,\vom}} \e^{-\int_0^s\kappa(\vx-\vom s')\dd s'} S(\vx-\vom s)\dd s\dd\omega
\\&
=
 \tfrac1{4\pi}\int_{\Omega} S(\vy)\frac{\e^{-\int_{[\vx,\vy]}\kappa}}{|\vy-\vx|^2}\dd\vy,
 \end{aligned}
\end{equation}
where $\vom'(\vom):= \vom-2(\vn\cdot\vom)\vn$ and  $\ds\int_{[\vx,\vy]}f:=|\vy-\vx|\int_0^1f(s\vy+(1-s)\vx)\dd s$.

To justify the last formula we refer to the following lemma with $\Psi(\vx,\vy)=S(\vy)\e^{-\int_{[\vx,\vy]}\kappa}$. Again, for clarity, we drop the first argument $\vx$.
\begin{lemma}\label{lem:one}
Let $\Omega$ be a convex bounded open set of $\R^3$; let $\Gamma$ be its boundary. Let $\Psi:\Omega\mapsto \R$ be continuous. Let $\tau_{\vx,\vom}\ge 0$ be such that $\vx-\tau_{\vx,\vom}\vom\in\Gamma$, $\vx\in\Omega$. Then 
\begin{equation*}
\int_\SS\int_0^{\tau_{\vx,\vom}}\Psi(\vx-\vom s)\dd s\dd\omega
=\int_{\Omega} \frac{\Psi(\vy)}{|\vy-\vx|^2}\dd\vy.
\end{equation*}
\end{lemma}
\begin{proof}:~
Denote $\tilde\Psi$ the extension of $\Psi$ by zero outside $\Omega$.
Let $\vom=(\cos\theta\sin\varphi,\sin\theta\sin\varphi,\cos\varphi)^T$, $\theta\in(0,2\pi)$, $\varphi\in(-\frac\pi 2,\frac\pi 2)$. Consider a partition of the semi infinite line starting at $\vx$ in direction $-\vom$ into segments of size $\delta s$ and denote $\vx_n= \vx-n\delta s \vom$. Then 
\begin{equation}\label{sunlight}
\begin{aligned}
\int_\SS\int_0^{\tau_{\vx,\vom}}&\tilde\Psi(\vx-\vom s)\dd s\dd\omega
=\lim_{\delta s\to 0}\sum_{n>0}\delta s\int_0^{2\pi}\int_{-\frac\pi 2}^{\frac\pi 2} \tilde\Psi(\vx_n)\cos\varphi\dd\varphi \dd\theta
\cr&
=\lim_{\delta s\to 0}\sum_{n>0}\int_0^{2\pi}\int_{-\frac\pi 2}^{\frac\pi 2} \frac{\tilde\Psi(\vx_n)}{|\vx-\vx_n|^2} |\vx-\vx_n|^2|\vx_{n+1}-\vx_n|\cos\varphi\dd\varphi \dd\theta.
 \end{aligned}
\end{equation}
We note that $|\vx-\vx_n|^2|\vx_{n+1}-\vx_n|\cos\varphi\dd\theta\dd\varphi$ is the elementary volume in the sector $\dd\theta\dd\varphi$ between the spheres centered at $\vx$ and of radii $|\vx-\vx_n|$ and $|\vx-\vx_{n+1}|$.  Therefore the right-hand side is an integral in $\vy\in\R^3$ of  $\frac{\tilde\Psi(\vy)}{|\vx-\vy|^2} |\vx-\vx_n|^2$.
\end{proof}
\begin{remark}\label{rem1}
When $\Omega$ is not convex, on may apply the lemma to its convex closure $\bar\Omega$ with $\kappa$ extended to $+\infty$ in $\bar\Omega\backslash\Omega$.
\end{remark}
\begin{remark}
When $R_\nu\equiv 0$, $S^E$ is given by  \eqref{gen1} with $Q_\nu$ in place of $I$. As \eqref{RTb}defines a map ${\cal T}:J\mapsto T$, 
\[
T(\vx)={\cal T}[J_\nu](\vx),~~\forall \vx\in\Omega,
\]
then, \eqref{eqS}, \eqref{gen1} is a nonlinear integral formulation for  $J$:
\begin{equation}
J_\nu(\vx)=S^E_\nu(\vx) + {\mathcal J}[\kappa(1-a)B_\nu({\cal T}[J_\nu])+\kappa a J_\nu](\vx),~~\forall \vx\in\Omega.
\end{equation}
\end{remark}
The following fixed point method  was shown in \cite{DIA} to be  monotone and convergent:
\begin{equation}
J^{k+1}_\nu(\vx)=S^E_\nu(\vx) + {\mathcal J}[\kappa(1-a)B_\nu({\cal T^k}[J^k_\nu](\vx))+\kappa a J^k_\nu](\vx),~~k=0,1,\dots
\end{equation}
Let us extend these properties to the RTE with RC.
For clarity let $\vx_\Sigma$ be short for $\vx_\Sigma(\vx,\vom)$ and  let 
\[
\vom'(\vom) :=\vom-2\vom\cdot\vn(\vx')\;\vn(\vx'), \quad 
\vx'_\Sigma := \vx_\Sigma(\vx_\Sigma(\vx,\vom),\vom')
\quad 
\text{ with }\vom:=\frac{\vx-\vx'}{|\vx-\vx'|}.
\]
Let us insert  \eqref{charact} and \eqref{RTc} in \eqref{gen1}. Then,
\begin{equation*}
\begin{aligned}&
 S_\nu^E(\vx)=S^E_{\nu,1}+S^E_{\nu,2}+S^E_{\nu,3} \hbox{ with }
\cr&
S^E_{\nu,1}(\vx) := \tfrac1{4\pi}\int_\SS Q_\nu(\vx_\Sigma,\vom)\e^{-\int_0^{\tau_{\vx,\vom}}\kappa(\vx-\vom s)\dd s}\dd\omega, 
 \cr & 
 S^E_{\nu,2}(\vx):= \tfrac1{4\pi}\int_\SS R_\nu(\vx_\Sigma,\vom) Q_\nu(\vx'_\Sigma,\vom') \left[\e^{-\int_0^{\tau_{\vx_\Sigma,\vom'}}\kappa(\vx_\Sigma-\vom' s)\dd s}
 \right. \cr & \hskip 8cm\left. 
 \e^{-\int_0^{\tau_{\vx,\vom}}\kappa(\vx-\vom s)\dd s}\right]\dd\omega,
\cr&
S^E_{\nu,3}(\vx):= \tfrac1{4\pi}\int_\SS\left[
   R_\nu(\vx_\Sigma,\vom)\e^{-\int_0^{\tau_{\vx,\vom}}\kappa(\vx-\vom s')\dd s'}
\right.
\\ & \hskip 4cm
\left.   
   \int_0^{\tau_{\vx_\Sigma,\vom'}} \e^{-\int_0^s\kappa(\vx_\Sigma-\vom' s')\dd s'} S_\nu(\vx_\Sigma-\vom' s)\dd s\; \right]\dd\omega.
 \end{aligned}
\end{equation*}
\begin{hypothesis}\label{H1}
Let us rule out multiple reflections and focal points:
\begin{enumerate}
\item
If $R_\nu(\vx_\Sigma(\vx,\vom),\vom)>0$, then $R_\nu(\vx_\Sigma(\vx_\Sigma(\vx,\vom),\vom'),\vom)=0$.
\item
Given $\vx$ and $\vy$, there is only a finite number $M$ of $\vx'_n\in\Gamma$ such that $[\vx'_n,\vy]$ is the reflected ray of $[\vx,\vx'_n]$. {\rm Note that $\vx'_n$ depends on $\vx$ and $\vy$.}
\end{enumerate}
\end{hypothesis}

\begin{proposition}\label{prop1}
Under Hypothesis \ref{H1}
\[
S^E_{\nu,3}(\vx):= \sum_{n=1}^M\tfrac{1}{4\pi}\int_\Omega R_\nu(\vx'_n,\tfrac{\vx-\vx'_n}{|\vx-\vx'_n|}) \frac{\e^{-\int_{[\vx,\vx'_n]\cup[\vx'_n,\vy]}\kappa}}{(|\vx-\vx'_n|+|\vx'_n-\vy|)^2}S(\vy)\dd\vy.
\]
\end{proposition}
\begin{proof}
Let $\vx(s):=\vx_\Sigma-\omega's$. By Lemma \ref{lem:one}, 
\begin{equation*}
\begin{aligned}
\int_\SS \int_0^{\tau_{\vx_\Sigma,\vom'}}&\left[
   R_\nu(\vx_\Sigma,\vom)\e^{-\int_{[\vx,\vx_\Sigma]\cup[\vx_\Sigma,\vx(s)]}\kappa}
   S(\vx(s))\dd s\; \right]\dd\omega 
\cr&
= \int_\Omega R_\nu(\vx_\Sigma,\vom) S(\vy)\frac{\e^{-\int_{[\vx,\vx_\Sigma ]\cup[\vx_\Sigma ,\vy]}\kappa}}{(|\vx-\vx_\Sigma |+|\vx_\Sigma -\vy|)^2}\dd\vy,
 \end{aligned}
\end{equation*}
provided that $[\vx_\Sigma,\vy]$ is reflected from $[\vx,\vx_\Sigma]$.
Now, by hypothesis, if $\vx$ and $\vy$ are given in $\Omega$ there are only a finite number of $\vx_\Sigma\in\Gamma$ for which $[\vx_\Sigma,\vy]$ is reflected from $[\vx,\vx_\Sigma]$, (see Figure \ref{sketch}).  
\end{proof}
\begin{proposition}\label{prop2}
Let Hypothesis \ref{H1} hold. Then the source terms from the boundaries are
\begin{align}& \label{SE1}
S^E_{\nu,1}(\vx) =
 \tfrac{1}{4\pi}\int_\Gamma Q_\nu(\vy,\tfrac{\vy-\vx}{|\vy-\vx|})
 \frac{[(\vy-\vx)\cdot\vn(\vy)]_-}{|\vy-\vx|^3}\e^{-\int_{[\vx,\vy]}\kappa}\dd\Gamma(\vy),
\\& \label{SE2}
S^E_{\nu,2}(\vx) =
 \sum_{n=1}^M\tfrac{1}{4\pi}\int_\Gamma  R_\nu(\vx'_n,\tfrac{\vx-\vx'_n}{|\vx-\vx'_n|})Q_\nu(\vy,\tfrac{\vx'_n-\vy}{|\vx'_n-\vy|})
 \cr&
 \hskip 4cm
 \frac{[(\vx'_n-\vy)\cdot\vn(\vy)]_- \e^{-\int_{[\vx,\vx'_n]\cup[\vx'_n,\vy]}\kappa}}{|\vx'_n-\vy|\;(|\vx-\vx'_n|+|\vx'_n-\vy|)^2}
 \dd\Gamma(\vy).
 \end{align}
{\rm  Recall that $\vx'_n$ depends on $\vy$.}
\end{proposition}
\begin{proof}:
Recall that a solid angle integral at $\vx$ of a surface $\Sigma$ is
\[
\int_\SS f(\vx,\vx')\dd\vom(\vx') = \int_\Sigma f(\vx,\vx')\frac{[(\vx-\vx')\cdot\vn(\vx')]_-}{|\vx-\vx'|}\frac{\dd\Sigma(\vx')}{|\vx-\vx'|^2}.
\]
Therefore, from the definition of $S^E_{\nu,2}$ above we see that \eqref{SE1} holds.

To prove \eqref{SE2} we start from the definition of $S^E_{\nu,2}$ above.  
 For clarity let us assume that $Q_\nu$ and $R_\nu$ do not depend on $\vom$.

Observe that if a ray from $\vx$ in the direction $-\vom$ does not hit, after reflection  at $\vx'$ on some $\Gamma_R$, a boundary $\Gamma_Q$ at $\vy$ where $Q_\nu(\vy)$ is non zero, then $\vom$ does not contribute to $S^E_{\nu,2}$. Thus, we can use the solid angle of $\Gamma_Q$. However the solid angle is not seen from $\vx$ but from $\bar\vx$, the symmetric of $\vx$ with respect to the tangent plane to $\Gamma_R$ at $\vx'$.  As the distance from $\bar\vx$ to $\vy$ is also $|\vx-\vx'|+|\vx'-\vy|$, we obtain \eqref{SE2}.
\end{proof}
\begin{corollary}
\begin{equation}
J_\nu(\vx) =\bar S^E_\nu(\vx) + \bar{\mathcal J}[S_\nu](\vx),
\end{equation}
with $\bar S^E_\nu(\vx):=S^E_{\nu,1}(\vx)+S^E_{\nu,2}(\vx)$ given by Proposition \ref{prop2} and 
\begin{equation}
\begin{aligned}&
\bar{\mathcal J}[S](\vx) = 
 \tfrac1{4\pi}\int_{\Omega} \left[\frac{\e^{-\int_{[\vx,\vy]}\kappa}}{|\vy-\vx|^2}
 +\sum_{n=1}^M \frac{\e^{-\int_{[\vx,\vx'_n]\cup[\vx'_n,\vy]}\kappa}}{(|\vx-\vx'_n|+|\vx'_n-\vy|)^2}
 R_\nu(\vx'_n,\tfrac{\vx-\vx'_n}{|\vx-\vx'_n|}) \right]S(\vy)\dd\vy.
\end{aligned}
\end{equation}
\end{corollary}
\subsection{Example}\label{exemple}
Assume that $\Gamma=\Gamma_Q\cup\Gamma_R$ and $Q_\nu(\vx,\vom)=[\vom\cdot\vn(\vx)]_-\; Q^0$ with $Q^0>0$ on $\Gamma_Q$ and $0$ on $\Gamma_R$. Assume $R_\nu(\vx,\vom)=R^0$ with $R_0>0$ on $\Gamma_R$ and $0$ on $\Gamma_Q$. Assume that there is never more than one reflection point on $\Gamma_R$, i.e. $M=1$. Then 
\begin{equation*}
\begin{aligned}&
\bar S^E_{\nu}(\vx) =
 \frac{Q^0}{4\pi}\int_{\Gamma_Q} \left[\left(\frac{[(\vy-\vx)\cdot\vn(\vy)]_-}{|\vy-\vx|^2}\right)^2\e^{-\int_{[\vx,\vy]}\kappa}\right.
\cr&
\hskip5cm
+\left. R^0
 \frac{([(\vx'_1-\vy)\cdot\vn(\vy)]_-)^2 \e^{-\int_{[\vx,\vx'_1]\cup[\vx'_1,\vy]}\kappa}}{|\vx'_1-\vy|^2\;(|\vx-\vx'_1|+|\vx'_1-\vy|)^2}\right]
 \dd\Gamma(\vy),
\cr&
\bar{\mathcal J}[S](\vx) = 
 \tfrac1{4\pi}\int_{\Omega} \left[\frac{\e^{-\int_{[\vx,\vy]}\kappa}}{|\vy-\vx|^2}
 +R^0 \frac{\e^{-\int_{[\vx,\vx'_1]\cup[\vx'_1,\vy]}\kappa}}{(|\vx-\vx'_1|+|\vx'_1-\vy|)^2}
   \right]S(\vy)\dd\vy.
 \end{aligned}
\end{equation*}

\begin{figure}[htbp]
\begin{center}
\begin{tikzpicture}[scale=0.5]
\draw  (-5,0) -- (5,0)--(4,2.5)--(-4,2.5)--(-5,0); 
\draw  (2,4.5) -- (5,4.5)--(4,6)--(1.5,6)--(2,4.5); 
\draw (-4,4) node [left] {$\vx$};
\draw[dashed]  (-4,4) -- (-1,1); 
\draw[dotted,blue]  (-4,-2) -- (1.5,6); 
\draw[dotted,blue]  (-4,-2) -- (2,4.5); 
\draw[dotted,blue]  (-4,-2) -- (5,4.5); 
\draw [blue] (-1.7,0.5) -- (-0.6,0.5)--(-1,1.2)--(-1.9,1.2)--(-1.7,0.5); 
\draw[red] (-4,-2) node [left] {$\bar\vx$};
\draw[dashed,red]  (-4,-2) -- (-1,1); 
\draw[thick,-stealth]  (-1,1) -- (-2,2);
\draw (-2,2) node [left] {$\vom$};
\draw[dotted, thick] (-1,1)--(-1,0);
\draw[thick,-stealth] (-1,0)--(-1,-1);
\draw (-1,1) node [left] {$\vx_\Sigma:=\vx_\Sigma(\vx,\vom)$};
\draw[thick,stealth-]  (-1,1) -- (0,2);
\draw (0,2) node [right] {$\vom'=\vom-2(\vn(\vx_\Sigma)\cdot\vom)\vn(\vx_\Sigma)$};
\draw (-1,-1) node [left] {$\vn(\vx_\Sigma)$};
\draw[dashed]  (-1,1) -- (3,5); 
\draw (3,5) node [right] {$\vx'_{\Sigma}:=\vx_\Sigma(\vx_\Sigma,\vom')$};
\draw[thick,-stealth]  (3,5) -- (2.7,7);
\draw (2.7,7) node [right] {$\vn(\vx'_{\Sigma})$};
\end{tikzpicture}
\caption{In this configuration the source $\Gamma_Q$ is the upper square. An RC is imposed on the lower plane $\Gamma_R$. $S^E_{\nu}$ has an integral of the solid angle of the upper square seen from $\vx$ plus an integral of the solid angle of the upper square seen from $\bar\vx$, the symmetric of $\vx$ with respect to $\Gamma_R$.}
\label{sketch}
\end{center}
\end{figure}
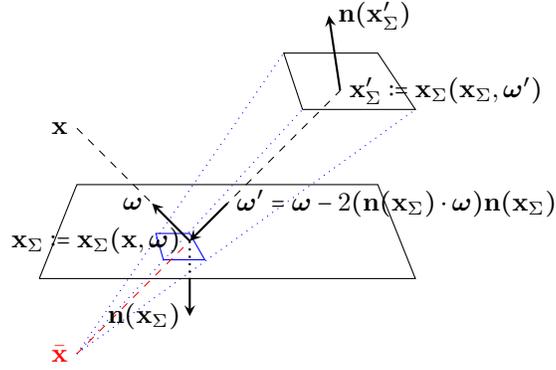

\subsection{Fixed Point Iterations}
Consider the fixed point iterations initialized with $T^0$ and $J^0=0$. %

\bigskip

\textbf{Algorithm}
For $k=0,1,\dots$:
\begin{equation}\label{algo}
\begin{aligned}&
\text{Set } S^k_\nu(\vx)=\kappa(1-a)B_\nu(T^k)+\kappa a J^k_\nu.
\\ &
\text{Set }  J^{k+1}_\nu(\vx)
=\bar S^E_\nu(\vx) + \bar{\mathcal  J}[S^k_\nu](\vx).
\\ &
\text{Compute $T^{k+1}$ by solving (using Newton algorithm) for each $\vx\in\Omega$ } 
\\ & \hskip 1cm 
\int_0^\infty{\kappa_\nu}(1-a_\nu)(J^{k+1}_\nu-B_\nu(T^{k+1}))\dd\nu=0.
 \end{aligned}
\end{equation}
\begin{proposition}
Let $\{J^*_\nu,T^*\}$ be the solution.   If $T^0(\vx)> T^*(\vx),~\forall\vx\in\Omega$ then the iterations are monotone decreasing: $T^{k}(\vx)> T^{k+1} >T^*(\vx),~\forall\vx\in\Omega$.  Conversely if  $T^0(\vx)< T^*(\vx),~\forall\vx\in\Omega$
then the iterations are monotone increasing: $T^{k}(\vx)< T^{k+1} <T^*(\vx),~\forall\vx\in\Omega$.
\end{proposition}

\begin{proof}:
Le us prove it for the monotone increasing sequence.

By subtracting the definition $J^k_\nu$ from that of $J^{k+1}_\nu$ and using the linearity of $\bar{\cal J}$, we obtain
\[
J^{k+1}_\nu(\vx)-J^{k}_\nu(\vx)=  \bar{\mathcal  J}[S^k_\nu-S^{k-1}_\nu](\vx).
\]
As ${\cal J}$ is a strictly positive operator, if $S^k_\nu>S^{k-1}_\nu$ for all $\vx$ then 
$J^{k+1}_\nu(\vx)>J^{k}_\nu(\vx)$.
The equation for $T^{k+1}$ is also monotone in the sense that 
\[
J^{k}_\nu(\vx)>J^{k}_\nu(\vx)\quad \implies \quad B_\nu(T^{k+1}) > B_\nu(T^{k})
\quad \implies \quad  T^{k+1}>T^k,
\]
because $B_\nu$ is increasing in $T$.

Conclusion: if $T^1>T^0$ and $S^1>S^0$ then $T^{k+1}>T^k$ for all $k$.  One sure way to impose it is to choose $T^0=0$ and $J^0=0$.

To prove that $T^k<T^*$ we observe that
\begin{equation*}
\begin{aligned}&
J^{k}_\nu(\vx)-J^{*}_\nu(\vx)=  \bar{\mathcal  J}[S^{k-1}_\nu-S^{*}_\nu](\vx).
\\&
\text{Hence } S^{k-1}_\nu<S^*_\nu  \quad \implies \quad 
J^{k}_\nu(\vx)<J^{*}_\nu(\vx) \quad \implies \quad T^k<T^*.
 \end{aligned}
\end{equation*}
\end{proof}
Discretization seems to preserve this property (see Figure \ref{converge}).
\begin{remark}
Henceforth, convergence and uniqueness can probably be proved as in \cite{FGOP3}, but there are technical difficulties of functional analysis which may not be appropriately discussed here.
\end{remark}

\section{FEM discretization and Compressed H-Matrices}
For clarity consider example \ref{exemple}.  As the values of $Q^0$ and $R^0$ take different values on $\Gamma_Q$ and $\Gamma_R$, we write $Q^0(\vx)$ and $R^0(\vx)$.
  
The domain $\Omega$ is discretized by a tetraedral mesh; the boundary $\Gamma$ is discretized by a triangular mesh, not necessarily conforming with the volume mesh.

Let  $\{\vx^j\}_1^N$ be the vertices of the tetraedra of $\Omega$ and $\{\tilde\vx^l\}_1^L$ the vertices of the triangles of $\Gamma$ . 

A continuous $P^1$ interpolation  of $J$ on the tetraedral mesh is:
\[
J(\vx)=\sum_1^N J_j {\hat w}^j(\vx)~\text{ where ${\hat w}^j$ is the $P^1$- Finite Element hat function of vertex $\vx^j$}.
\]
 Then 
\begin{equation*}
\begin{aligned}&
S_{\nu,j}:= a J_{\nu,j}+(1-a)B_\nu(T_j), \quad 
J_{\nu,i}:=\bar S^E_{\nu,i}+\sum_j G_{\kappa}^{ij}S_{\nu,j}\quad \text{ where } 
\\&
 G_{\kappa}^{ij}=\tfrac1{4\pi}\int_\Omega\left[ \kappa\frac{\e^{-\int_{[\vx^i-\vy]}\kappa}}{|\vx^i-\vy|^2}\dd y 
+ \sum_{n=1}^M R^0(\vx'_n) \frac{\e^{-\int_{[\vx^i,\vx'_n]\cup[\vx'_n,\vy]}\kappa}}{(|\vx^i-\vx'_n|+|\vx'_n-\vy|)^2}\right]{\hat w}^j(\vy)\dd\vy
 \\ & \text{ and where }
 \bar S^E_{\nu,i} =
  \tfrac{1}{4\pi}\int_\Gamma Q^0(\vy)\left[
 \left(\frac{[(\vx^i-\vy)\cdot\vn(\vy)]_-}{|\vx^i-\vy|^2}\right)^2\e^{-\int_{[\vx^i,\vy]}\kappa}\right.
 \\ & \left.  \hskip 3cm
+ \sum_{n=1}^M R^0(\vx'_n)
 \frac{([(\vx'_n-\vy)\cdot\vn(\vy)]_-)^2 \e^{-\int_{[\vx^i,\vx'_n]\cup[\vx'_n,\vy]}\kappa}}{|\vx'_n-\vy|^2(|\vx^i-\vx'_n|+|\vx'_n-\vy|)^2}
\right] \dd\Gamma(\vy).
\end{aligned}
\end{equation*}
The integrals are approximated with quadrature at points $\{\vx^j_q\}_1^{M_q}$.  The points are inside the elements; consequently $|\vx^i-\vx_q^j|$ is never zero. A formula of degree 5, with $M_q=14$, is used when $|\vx^i-\vy|$ is small and of degree 2, with $M_q=4$, otherwise; the results do not change when higher degrees are used. Fortunately when $\vx^i$ is close to $\Gamma$ an analytical formula can be used \cite{FGOP3}.

To compute $\vx'_n$ such that $[\vy,\vx'_n]$ is the reflected ray of $[\vx'_n,\vx^i]$ a loop on all the elements of the reflecting boundaries is necessary. This can be expensive, but in the case of planar reflective boundaries the symmetric point $\bar\vx^i$ is easy to compute and so is the intersection of $[\bar\vx^i,\vy]$ with the reflective boundary.

Finally, to the vector $\{\bar S^E_{\nu,i}\}_{i=1}^N$ we associate a matrix $\{\bar S^E_{i,l}\}_{i,l=1}^{N,L}$ by replacing $Q^0(\vy)$ above by $\tilde w^l(\vy)$. Then:
\[
Q^0(\vy)=\sum_1^L Q^0_l \tilde w^l(\vy)
\implies
\bar S^E_{\nu,i} = \sum_1^L \bar S^E_{i,l} Q^0_l.
\]

\subsection{Compression}

For each $\nu$ we have two large dense matrices, $\{\bar G_{i,j}\}_{i,j=1}^{N,N}$ and $\{\bar S^E_{i,l}\}_{i,l=1}^{N,L}$.
\begin{remark}
Note that for each value of $\nu$ two  matrices are needed. However on close inspection it is really two matrices for each value of $\kappa_\nu$.  Very often, less than ten values are sufficient to represent a general $\kappa_\nu$ by a piece-wise constant interpolation on these values.
\end{remark}

These matrices can be compressed as ${\mathcal H}$-matrices \cite{beben},\cite{rjasanow},\cite{sauter} (and the references therein)  so that the matrix-vector product has complexity $O(N\ln N)$.
 
 The method works best when the kernel  in the integrals decays with the distance between $\vx^i$ and $\vy$. In all matrices the kernel decays with the square of the distance.
The ${\mathcal H}$-matrix  approximation views ${\bf G}$ as a hierarchical tree of square blocks.
The blocks correspond to interactions between clusters of points near $\vx^j$ and near $\vx^i$. A far-field interaction block can be approximated by a low-rank matrix because its singular value decomposition (SVD) has fast decaying singular values.
 We use the {\it Partially Pivoted Adaptive Cross-Approximation} (ACA) \cite{borm} to approximate the first terms of the SVD of the blocks, because only $r$ rows and $r$ columns are needed instead of the whole block, where $r$ is the rank of the approximation.  The rank is a function of a user defined parameter $\epsilon$ connected to the relative  Frobenius norm error. Another criterion must be met: if $R_1$ (resp. $R_2$) is the radius of a cluster of points centered at $\vx_1$ (resp. $\vx_2$), then one goes down the hierarchical tree  until the corresponding block  satisfies $ \max(R_1,R_2)<\eta |\vx_1-\vx_2|$ where $\eta$ is a user defined parameter. If a leaf is reached, the block is not compressed and all the elements are computed.

The precision is not guaranteed  if $[(\vx-\vy)\cdot\vn(\vy)]_-$  jumps from one triangular face to another is large. 
A similar singularity caused by normals is analyzed for a double layer potential formulation in \cite{beben} (Example 3.38, p.148) and a remedy is proposed.
To check whether this remedy is needed here we ran two  cases, one without compression and one with 97\% compression. No difference was observed.

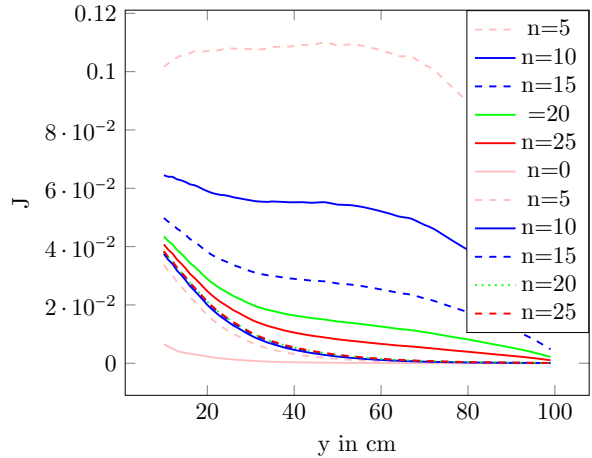
\begin{figure}[H]
\begin{center}
\begin{tikzpicture}[scale=0.9]
\begin{axis}[legend style={at={(1,1)},anchor=north east}, compat=1.3,
   xlabel= {y in cm},
  ylabel= {J}
  ]
%
%
\addplot[thick,dashed,color=pink,mark=none, mark size=1pt] table [x index=0, y index=1]{fig/koba4T5R.txt};
\addlegendentry{n=5}
\addplot[thick,solid,color=blue,mark=none, mark size=1pt] table [x index=0, y index=1]{fig/koba4T10R.txt};
\addlegendentry{n=10}
\addplot[thick,dashed,color=blue,mark=none, mark size=1pt] table [x index=0, y index=1]{fig/koba4T15R.txt};
\addlegendentry{n=15}
\addplot[thick,solid,color=green,mark=none, mark size=1pt] table [x index=0, y index=1]{fig/koba4T20R.txt};
\addlegendentry{=20}
\addplot[thick,solid,color=red,mark=none, mark size=1pt] table [x index=0, y index=1]{fig/koba4T25R.txt};
\addlegendentry{n=25}
\addplot[thick,solid,color=pink,mark=none, mark size=1pt] table [x index=0, y index=1]{fig/koba40R.txt};
\addlegendentry{n=0}
\addplot[thick,dashed,color=pink,mark=none, mark size=1pt] table [x index=0, y index=1]{fig/koba45R.txt};
\addlegendentry{n=5}
\addplot[thick,solid,color=blue,mark=none, mark size=1pt] table [x index=0, y index=1]{fig/koba410R.txt};
\addlegendentry{n=10}
\addplot[thick,dashed,color=blue,mark=none, mark size=1pt] table [x index=0, y index=1]{fig/koba415R.txt};
\addlegendentry{n=15}
\addplot[thick,dotted,color=green,mark=none, mark size=1pt] table [x index=0, y index=1]{fig/koba420R.txt};
\addlegendentry{n=20}
\addplot[thick,dashed,color=red,mark=none, mark size=1pt] table [x index=0, y index=1]{fig/koba425R.txt};
\addlegendentry{n=25}
\end{axis}
\end{tikzpicture}
\caption{\label{converge}  Values of $J$, for the academic test, along the $y$ axis at $x=z=15$ computed with a RC.  Convergence versus iteration number n. When the  scaled temperature is initialized to $T^0=0.001$ at $n=0$ the convergence is monotonously increasing. When $T^0=0.44$ the convergence is monotonously decreasing. }
\end{center}
\end{figure}
\section{{An Academic Test}}
In \cite{koba} a semi-analytic solution of the RTE is given for a geometry shown on Figure \ref{kobayashifig}. In this test $a=0$ and $\kappa$ is a function of $\vx$ but not of $\nu$ . Hence the grey formulation can be used for $\bar I=\int_0^\infty I_\nu\dd\nu$.
By averaging \eqref{algo} in $\nu$ and due to the Stefan-Boltzmann relation, the following holds: 
\begin{equation}
\begin{aligned}&
 \label{grey}
\int_0^\infty B_\nu(T)\dd\nu=\sigma T^4 \hbox{ with }\sigma=\frac{\pi^4}{15}\quad \implies
\\ &
 \bar J^{k+1}(\vx)
=\bar S^E(\vx) + \bar{\mathcal  J}[\kappa\sigma (T^k)^4](\vx),
\quad
\kappa\sigma (T^{k+1})^4 = \kappa\bar J^{k+1}.
\end{aligned}
\end{equation}

\subsection{The Geometry}
The outer container is $D=(0,60)\times(0,100)\times(0,60)$, in cm.  A  cube  $C=[0,10]^3$, inside $D$,  radiates with intensity $Q^0=0.1$. A rectangular cylinder prolonging the radiating cube $(0,10)\times(10,100)\times(0,10)$ has a low absorption $\kappa=10^{-4}$ while the rest has $\kappa=0.1$.
In Kobayashi's test case 1A there is no scattering, and the three planes containing the origin reflects the radiations perfectly ($R_\nu=1$): $(O,x,z)$, $(O,x,y)$, $(O,y,z)$.  

Unfortunately the present method cannot handle volumic radiating region.  Consequently we have kept the geometry but only the 3 faces of $C$ inside $D$ radiate in all directions $\vom$ with intensity $Q^0[\vom\cdot\vn]_-$, where $\vn$ is the normal to the cube's face pointing inside the cube. The domain is $\Omega=D\backslash C$ (see Figure \ref{kobayashifig}).
We refer to this case as Test-3.

  \begin{wrapfigure}{r}{0.4\textwidth}
  \centering
\includegraphics[width=0.4\textwidth]{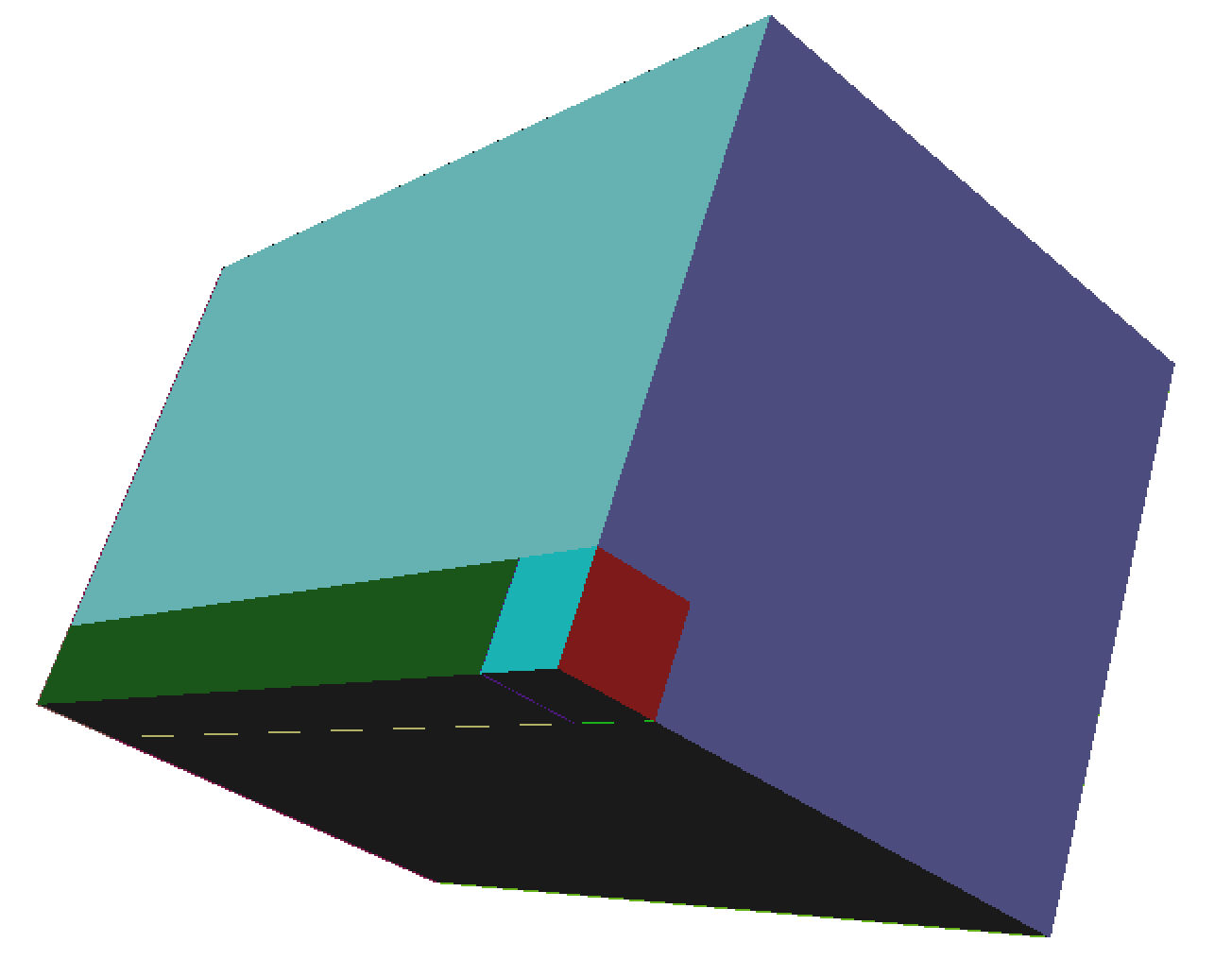}
    \caption{ A small cube (colored blue and red on the figure) radiates normally to its faces in a medium which has a very small absorption coefficient $\kappa=10^{-4}$ in the  cylinder prolonging the cube and $\kappa=0.1$ elsewhere. }
    \label{kobayashifig}
\end{wrapfigure}

\subsection{Results}

To assert the precision of the method we consider first only one reflective plane, $\Gamma_R=(0,y,z)$ and a constant $\kappa=0.1$. We refer to this case as Test-1.   Test-2 is Test-1 with  $\kappa$ is as in Test-3. 

First we verify, on Test-3, that the convergence is monotone increasing if $T^0$ is small and monotone decreasing if $T^0$ is large (Figure \ref{converge}). Note that the monotone increasing sequence converges faster.

Next, we compare the results, on Test-1, with a computation on a domain $\bar D=(-60,60)\times(0,100)\times(0,60)$ which is $D$ plus the symmetric of $D$ with respect to the plane $(0,y,z)$. This is because reflection on a plane is equivalent to extending the domain  by symmetry with respect to that plane.

Figure \ref{double} shows, on Test-1, level surfaces of $J$ computed on the symmetrized domain (but restricted to the original domain) and compared with the same level surfaces but computed with the RC. Surfaces with similar colors should be near each other.  In fact the difference is not visible except near $z=0$ .

 Figure \ref{doublenoR} shows, on Test-2, the level surfaces of $J$ computed with the RC,  and the same level surfaces but computed without any RC on the $(O,y,z)$ plane.  It is seen that surfaces with similar color are far from each other.
By comparing Figure \ref{double} with Figure \ref{doublenoR} we see that the RC does almost the same as symmetry and that no RC at all is a non viable approximation for this problem.  
\vskip-0.3cm
\begin{figure}[H]
  \begin{center}
\includegraphics[width=10cm, clip, trim = {5cm 5cm 0cm 5cm}]{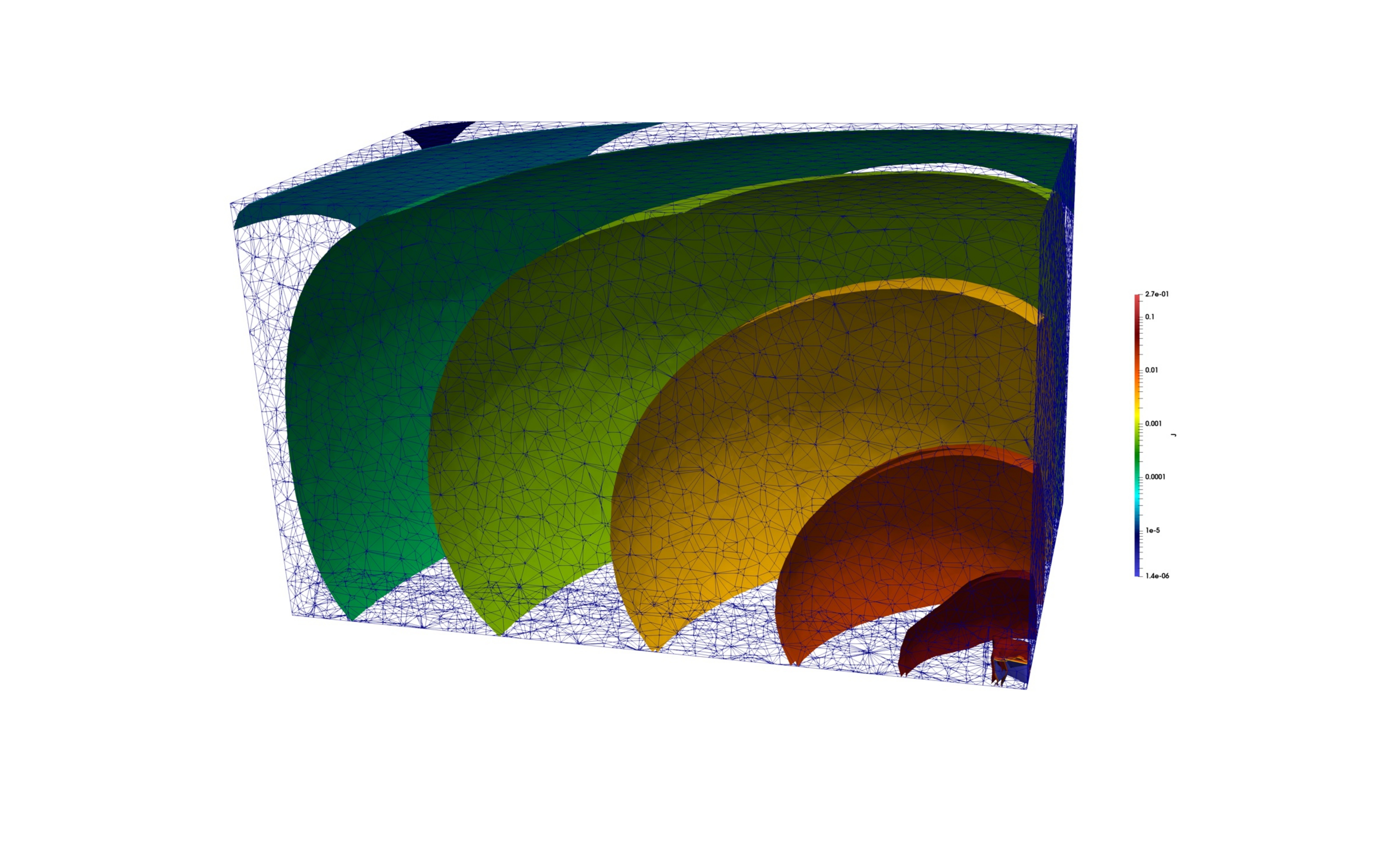}
\caption{Level surfaces of $J$ using a log scale computed with $\kappa=0.1$ and only one reflective plane, ($0,y,z)$ facing us, slightly to the left. Comparison between a computation done with the RC and a computation done on a symmetrized domain, double in size.Surfaces of equal colors are so near each other that it is hard to distinguish them.}
  \label{double}
  \end{center}
\end{figure}
\vskip-1cm
 \begin{figure}[H]
  \begin{center}
\includegraphics[width=9cm, clip, trim = {5cm 5cm 0cm 5cm}]{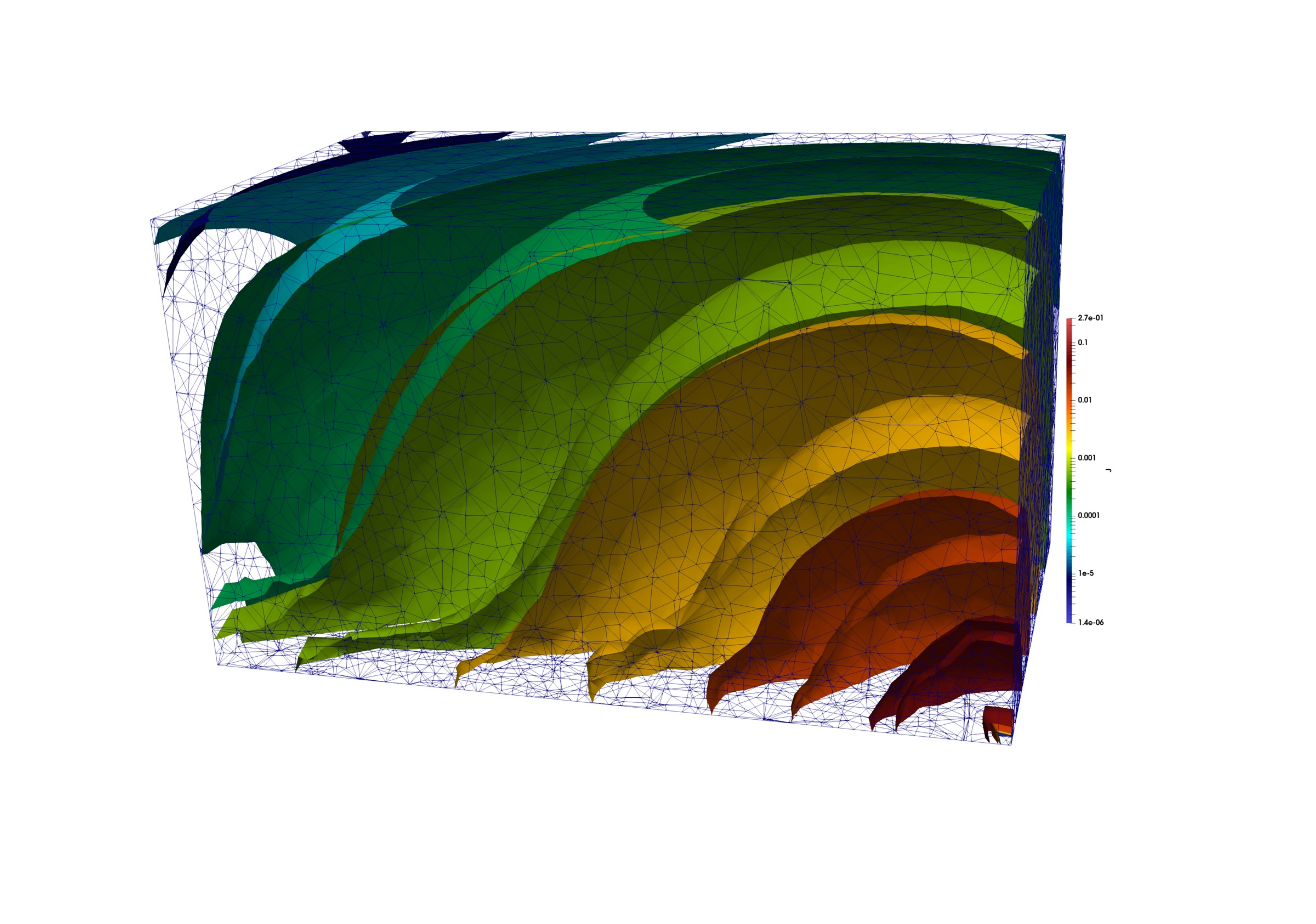}
    \caption{Same as in Figure \ref{double} but the RC is not used in one computation. Surfaces of equal colors are far from those using the RC, indicating the absolute necessity of a RC. Here $\kappa$ is as in Test-3.}
  \label{doublenoR}
  \end{center}
\end{figure}
Similarly, Figure \ref{compare1Da} shows $x\mapsto J(x,25,25)$, computed on the symmetrized domain, or with the radiative condition or without it.

Finally, Figure \ref{compare1Dc} shows $x\mapsto J(x,25,25)$, computed with the RC on 4 meshes. The same first 4  meshes are used in Table \ref{CPU} where the theoretical complexity $N\sqrt[3]{N}\ln N$ is approximately observed.  The compressing ratio for the surface and the volume matrices are shown too.

\begin{figure}[H]
\begin{minipage} [b]{0.5\textwidth} 
\begin{center}
\begin{tikzpicture}[scale=0.75]
\begin{axis}[legend style={at={(1,1)},anchor=north east}, compat=1.3,
   xlabel= {x in cm},
  ylabel= {J}
  ]
\addplot[thick,dashed,color=blue,mark=none, mark size=1pt] table [x index=0, y index=1]{fig/koba5R.txt};
\addlegendentry{With RC}
\addplot[thick,dotted,color=red,mark=none, mark size=1pt] table [x index=0, y index=1]{fig/koba5S.txt};
\addlegendentry{Symmetrized domain}
\addplot[thick,solid,color=magenta,mark=none, mark size=1pt] table [x index=0, y index=1]{fig/koba5N.txt};
\addlegendentry{No RC}
\end{axis}
\end{tikzpicture}
\caption{\label{compare1Da}  Values of $J$ along the $x$ axis at $y=z=25$ computed by different methods on the coarse mesh for Test-1.}
\end{center}
\end{minipage}
\hskip0.5cm
\begin{minipage} [b]{0.5\textwidth} 
\begin{center}
\begin{tikzpicture}[scale=0.75]
\begin{axis}[legend style={at={(1,1)},anchor=north east}, compat=1.3,
   xlabel= {x in cm},
  ylabel= {J}
  ]
\addplot[thick,dashed,color=black,mark=none, mark size=1pt] table [x index=0, y index=1]{fig/koba2R.txt};
\addlegendentry{N=84042}
\addplot[thick,solid,color=magenta,mark=none, mark size=1pt] table [x index=0, y index=1]{fig/koba3R.txt};
\addlegendentry{N=26189}
\addplot[thick,solid,color=blue,mark=none, mark size=1pt] table [x index=0, y index=1]{fig/koba5Ra.txt};
\addlegendentry{N=8003}

\addplot[thick,solid,color=green,mark=none, mark size=1pt] table [x index=0, y index=1]{fig/koba7R.txt};
\addlegendentry{N=2758}
\end{axis}
\end{tikzpicture}
\caption{\label{compare1Dc}  Values of $J$ along the $x$ axis at $y=z=25$ computed on Test-3 with different meshes.}
\end{center}
\end{minipage}
\end{figure}

\begin{table}[htp]
\caption{CPU time on a MacBookPro M1, and compression level (C.L.)for Test-3.}
\begin{center}
\begin{tabular}{|c|c|c|c|c|}
N vertices & Surface C.L. & Volume C.L. & CPU &$\frac{10^5 CPU}{N\sqrt[3]{N}\log N}$\cr
\hline
2758 & 0.43 & 0.60 & 5.3"  & 1.73\cr
8003 & 0.56 & 0.77 & 16.5"  & 1.14\cr
26189 & 0.67 & 0.89 & 79.4"  & 1.00 \cr
84042 & 0.75 & 0.95 & 389" & 0.93 \cr
195974 & 0.82 & 0.97 & 1563" & 1.13\cr
\end{tabular}
\end{center}
\label{CPU}
\end{table}%

\medskip
Test 1A of \cite{koba}, denoted here Test-3, has been computed, i.e. non constant $\kappa$ and 3 reflective planes. The surface levels of $J$ are shown on Figure \ref{doublenoR2}. Convergence versus mesh size  on the line $(x,25,25)$ is shown on Figure \ref{compare1Dc}.

 The comparison with the data in \cite{koba} on the line $(5,y,5)$ is shown on Figure \ref{converge2}. But since the radiative sources are different (volumic in Kobayashi's and surfacic in our case) we have scaled the result with Kobayashi's value at $x=5,y=15,z=5$.

Finally the $L^2$ error is computed by using the finest mesh, $N=195974$ as a reference solution.  The results are displayed on Figure \ref{errorkoba}.  It shows the $L^2$-error versus $h:=\sqrt[-3]{N}$, in log-log scales.

\begin{figure}[H]
\begin{center}
\begin{minipage} [b]{0.45\textwidth}
\begin{tikzpicture}[scale=0.7]
\begin{axis}[legend style={at={(1,1)},anchor=north east}, compat=1.3,
   xlabel= {y in cm},
  ylabel= {J}
  ]
\addplot[thick,dotted,color=blue,mark=*, mark size=1pt] table [x index=0, y index=1]{fig/kobayashi2.txt};
\addlegendentry{Kobayashi}
\addplot[thick,solid,color=red,mark=none, mark size=1pt] table [x index=0, y index=3]{fig/kobayashi2.txt};
\addlegendentry{Present method}
\end{axis}
\end{tikzpicture}
\caption{\label{converge2}  Values of $J$ versus $y\ge 15$ at $x=z=5$ and comparison with the values given in \cite{koba}. A scaling is applied so that the radiation intensities coincide at  $y=15$ (because \cite{koba} is given for volumic source and the present method handles only surface sources). }
\end{minipage}
\hskip 1cm
\begin{minipage} [b]{0.45\textwidth} 
\begin{tikzpicture}[scale=0.7]
\begin{axis}[legend style={at={(1,1)},anchor=north east}, compat=1.3,
   xmax=1.7, xmin=1.1,
   xlabel= {$x=\log\sqrt[3]{N}$},
  ylabel= {$y=\log$-error}
  ]
\addplot[thick,dotted,color=blue,mark=*, mark size=1pt] table [x index=0, y index=1]{fig/errorkoba3.txt};
\addlegendentry{Kobayashi $L^2$-error}
\addplot[thick,dotted,color=red,mark=none, mark size=1pt]{-x+1.8};
\addlegendentry{$y=-x+1.8$}
\end{axis}
\end{tikzpicture}
\caption{\label{errorkoba}  Log-log plot of $L^2$ error versus the average mesh size $h=\sqrt[3]{N}$ for Test-3. The line $-x+1.8$ indicates an error $O(h)$. The reference solution is computed on a mesh with $N=195974$. The plotted points are computed on meshes with $N$ as in Table \ref{CPU}.}
\end{minipage}
\end{center}
\end{figure}

 \begin{figure}[H]
  \begin{center}
\includegraphics[width=14cm, clip, trim = {5cm 5cm 0cm 5cm}]{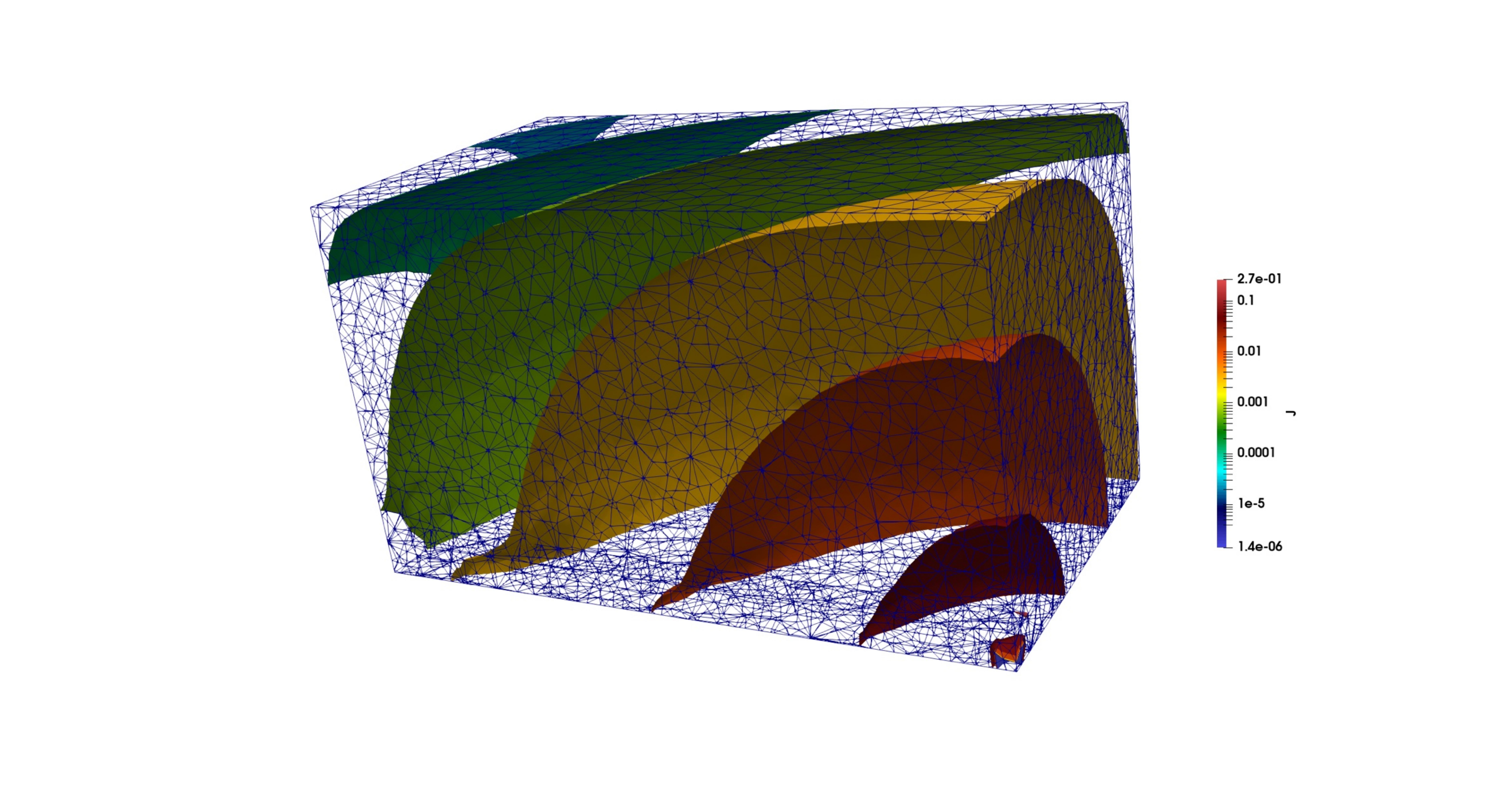}
\vskip-1cm
    \caption{Kobayashi's test: Level surfaces of $J$ using a log scale. The reflective planes are the $(O,x,z),(O,y,z),(O,x,y)$. The origin $O$ is the lower right corner.}
  \label{doublenoR2}
  \end{center}
\end{figure}

\section{The Chamonix Valley}

In \cite{JCP} the temperature in the Chamonix Valley due to sunlight was studied.  With units in 10km, the emitting domain (the ground) is a rectangle $[-0.2,3.32]\times[-3.35,0.163]$, with the Chamonix city at $(1.5,-1.5)$. The 3D domains is the emitting domain extruded above the ground up to $z=1$, i.e. 10km altitude.
The Mont Blanc, in the lower left part of the map is 4807m high. The domain is discretized into tetraedra by a 3D automatic mesh generator from a surface mesh.

Naturally the results are affected by the domain truncation because points near the boundaries receive less than half the scattered light.  Now RC can be applied to the 4 vertical planes of the truncation. Note that near the corners there is still a light deficiency which could be corrected by allowing multiple reflections (a programming challenge).

\subsection{Settings}
The ground surface radiates, proportionally to the vertical component of the normal $\vn_z$, the light of a black body at temperature 300°C at all frequencies (but mostly infrared). The intensity was set at $Q^0=2.5$ so as to obtain meaningful temperatures, but since the Earth is not in thermal equilibrium with the sunlight it receives, this choice is arbitrary.  In any case rescaling is easily done as $J$ is proportional to $Q^0$. 

All mountains are covered with snow above 2500m.  The snow-covered ground emits only $0.3Q^0$.

The ground surface mesh has 95K vertices and the volume mesh has 786K vertices.  In general 12 fixed point iterations are sufficient to find the temperature $T$ from $J_\nu$ and these decrease the error by 6 orders of magnitude.

The surface-to-volume matrix  compressed to level 0.942. The volume-to-volume matrix compressed to level 0.982.

\subsection{Test 1: the Grey Case}

In this test $\kappa$ depends on the altitude, but not on $\nu$: $\kappa=\tfrac12(1-a z)$, with $a=\tfrac34$, except in the cloud. The cloud is a layer between altitude $z_m=0.3$, i.e; 3000m and $z_M=0.7$, i.e.7000m where $\kappa$ is multiplied by a Gaussian random number of means 0.2 and variance 0.8.
Scattering is only in the cloud with $a=0.3(z-z_m)_+(z_M-z)_+/(4(z_M-z_m)^2$.

\smallskip

The program ran on a French national supercomputer in 5'42" using 1920 processors and 12 threads per MPI process. The surface-to-volume matrix was constructed in 38.5" with RC and 13.53 without. The volume-to-volume matrix took 145" with RC and 95.6 without. 
\vskip-1cm
\begin{figure}[H]
\begin{center}
\vskip-1cm
\includegraphics[width=6cm, clip, trim = {5cm 5cm 0cm 5cm}]{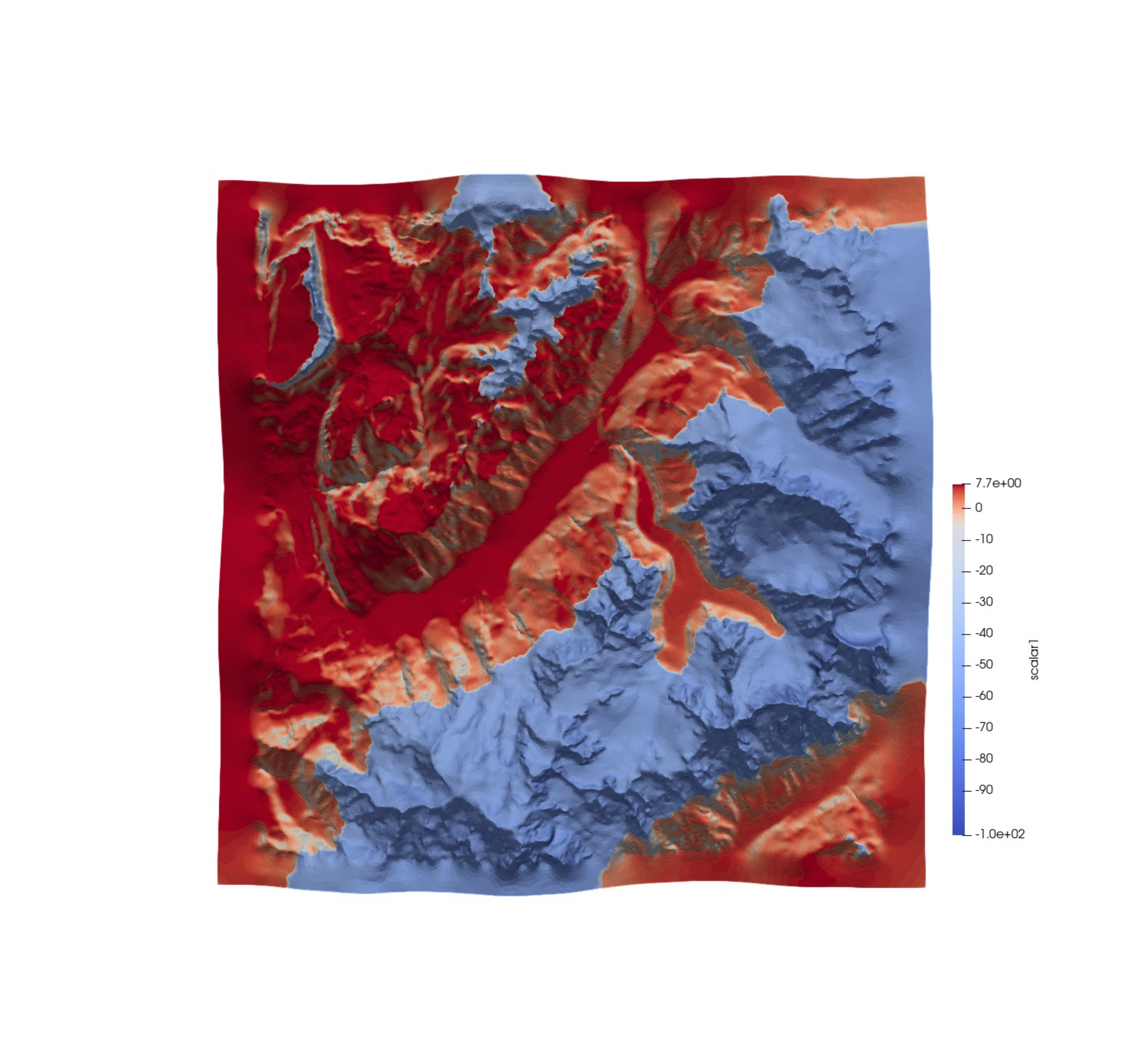}
\includegraphics[width=6cm, clip, trim = {5cm 5cm 0cm 5cm}]{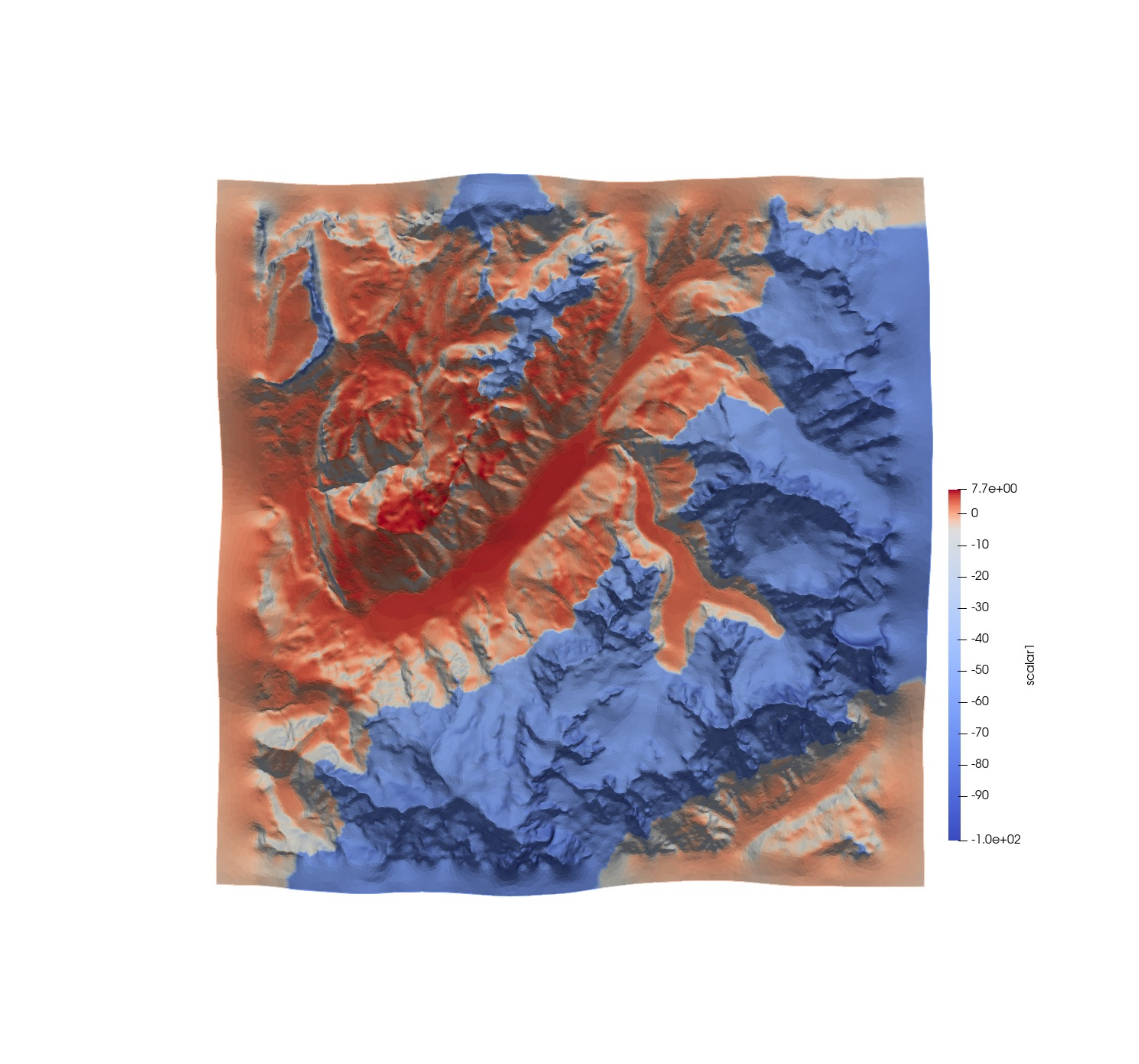}
\includegraphics[width=6cm, clip, trim = {5cm 5cm 0cm 5cm}]{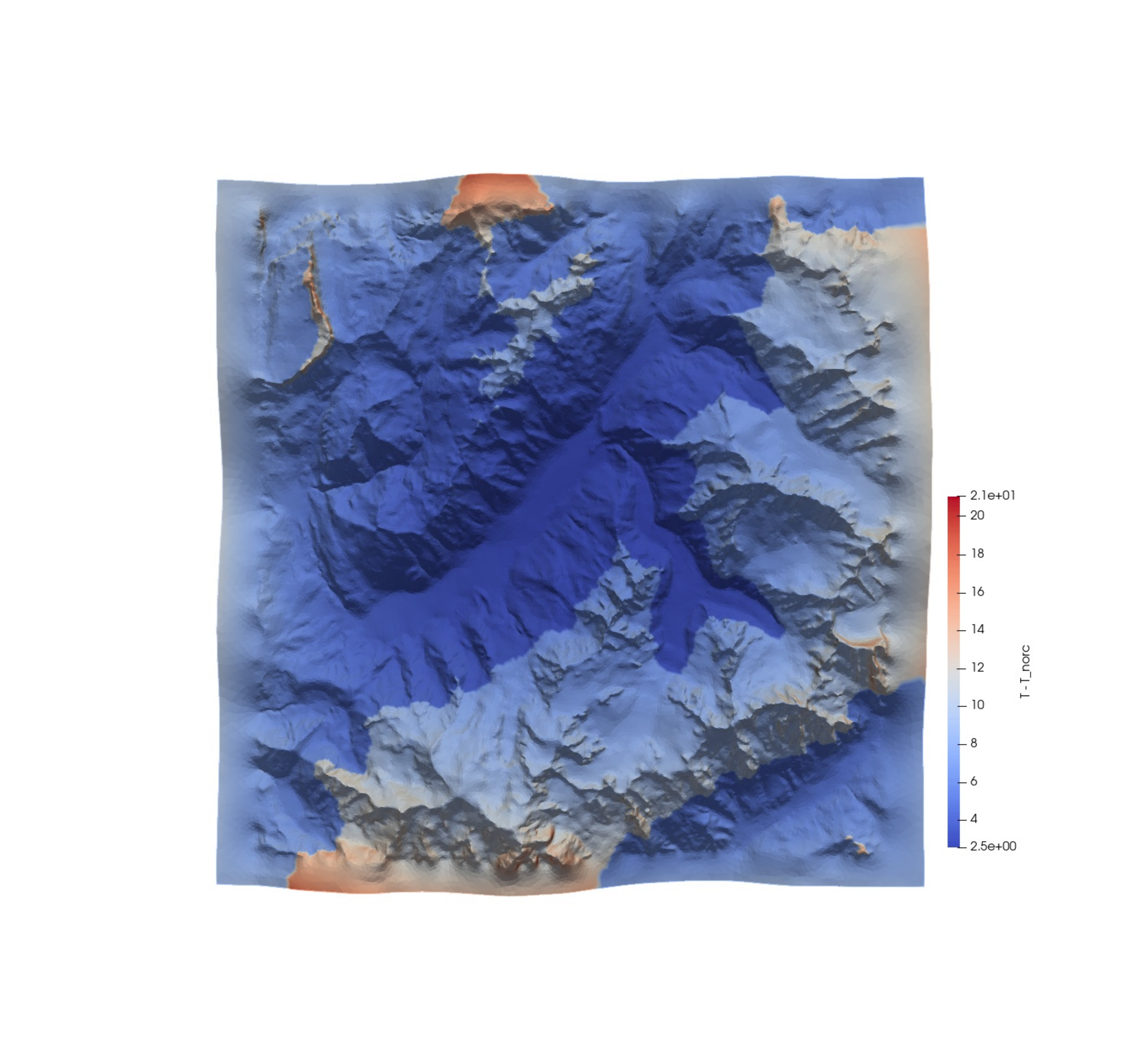}
\caption{Ground temperatures (in $^o$C) computed with RC on the 4 vertical boundaries (left) and without them (right). The third figure displays the difference $T$ with RC minus $T$ without RC.}
\label{cham3D}
\end{center}
\end{figure}
%
%
%
\begin{figure}[H]
\begin{center}
\includegraphics[width=12cm]{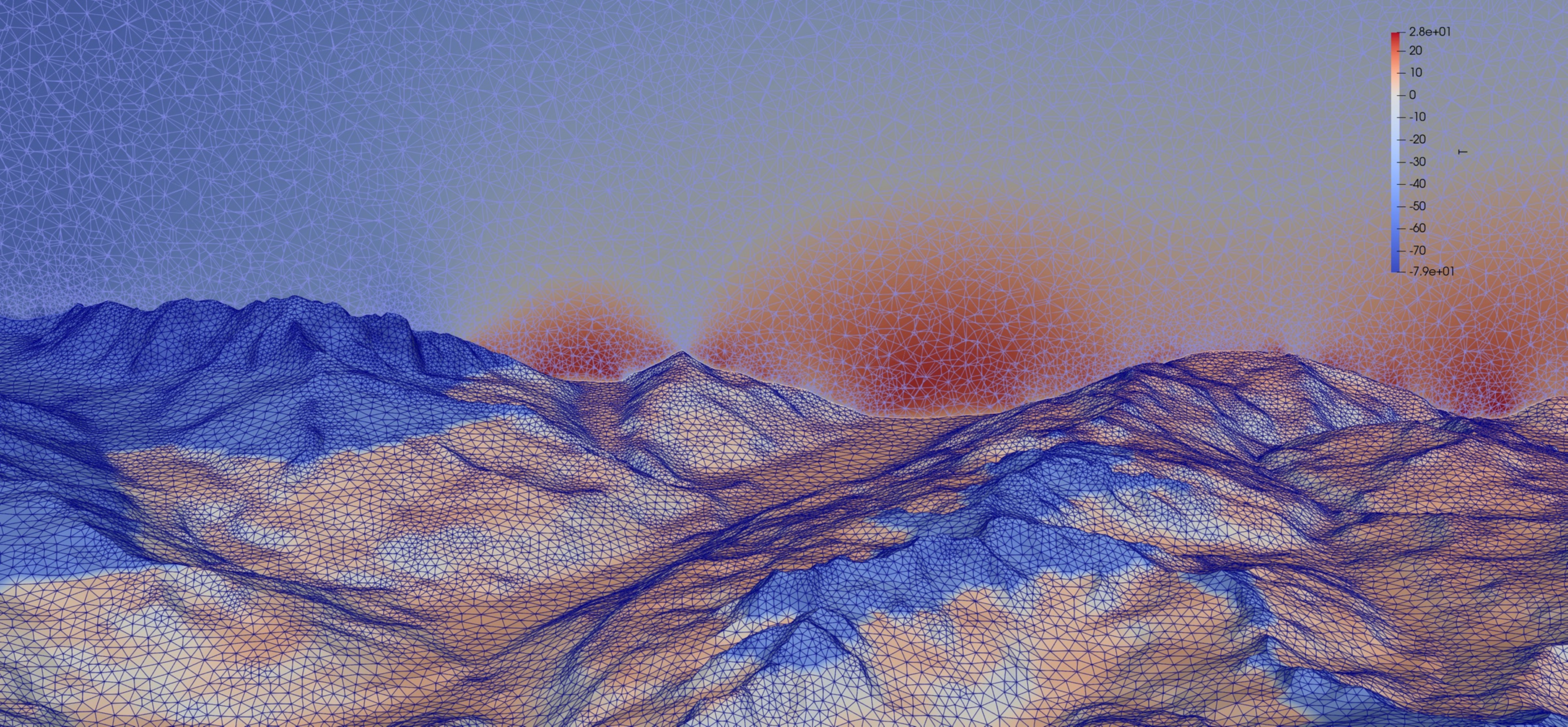}
\caption{Ground and vertical temperatures (in $^o$C) computed with RC in the valley of Chamonix. The mesh is shown in blue on the ground and the intersection of the mesh with the vertical plane is shown in white.}
\label{cutCham}
\end{center}
\end{figure}

Figure \ref{cham3D} shows the computed temperatures (in Celcius) on the ground with and without RC.  The difference between the two temperature fields is also displayed: it is noticeably hotter everywhere by a few degrees when computed  with RC.
Finally, Figure \ref{cutCham} shows the mesh and the temperatures on the ground and on a vertical plane accross the Chamonix valley.  The parabolic shape of the mountains increases the temperatures above the ground.  

\subsection{Test 2: the General Case}
The setting does not change except that now  $\kappa$ varies with altitude as before but with $a=\tfrac34$ and its dependance on $\nu$ is read from the Gemini website \cite{gemini} (see \cite{JCP} for details).  583 points are needed to discretize $\nu$, but only 8 values are retained for a piecewise discontinuous approximation of $\kappa_\nu$ in the exponentials in the matrices. 
The computing time is roughly 8 times that of the grey case.

The temperature versus altitude above the Chamonix city is plotted on Figure \ref{Tz}.  The sudden temperature increase just above the ground is persistent with mesh refinement near the ground. The same computation was done in the same domain with the same mesh but on a flat ground $z=0$ (the domain is parallelipedic). Then there is no sudden increase, implying that the sudden increase is due to the radiation in a U-shaped valley.
%
%
\begin{figure}[H]
\begin{center}
\begin{minipage} [b]{0.45\textwidth}
\begin{tikzpicture}[scale=0.7]
\begin{axis}[legend style={at={(1,1)},anchor=north east}, compat=1.3,
   xmin=0.1,
   xlabel= {Altitude ($1=10\,000$m)},
  ylabel= {T ($^oC$)}
  ]
\addplot[thick,solid,color=red,mark=none, mark size=1pt] table [x index=0, y index=1]{fig/noonScloudScattKtempe2.txt};
\addlegendentry{Test2}
\addplot[thick,dotted,color=red,mark=none, mark size=1pt] table [x index=0, y index=1]{fig/noonScloudScattKKtempe2.txt};
\addlegendentry{Test2+\texttt{CO}$_2$}
\addplot[thick,solid,color=blue,mark=none, mark size=1pt] table [x index=0, y index=1]{fig/noonScloudScattKtempe.txt};
\addlegendentry{Flat 3D}
\addplot[thick,dotted,color=blue,mark=none, mark size=1pt] table [x index=0, y index=1]{fig/noonScloudScattKKtempe.txt};
\addlegendentry{Flat 3D+\texttt{CO}$_2$}
\addplot[thick,solid,color=green,mark=none, mark size=1pt] table [x index=0, y index=1]{fig/temperature00.txt};
\addlegendentry{Flat 1D}

\end{axis}
\end{tikzpicture}
\caption{\label{Tz}  Values of $T(1.5,-1.5,z)$ versus $z$ for Chamonix and for a flat ground. In both case, the dotted curves are the results with $\kappa$ modified by adding \texttt{CO}$_2$.}
\end{minipage}
\hskip 0.5cm
\begin{minipage} [b]{0.45\textwidth}
\begin{tikzpicture}[scale=0.7]
\begin{axis}[legend style={at={(1,1)},anchor=north east}, compat=1.3,
   xmin=0,xmax=40,ymax=45,
   xlabel= {wavelength $3/\nu$},
  ylabel= {\color{red}$10\times A_\nu$~~~~~~~~\color{blue}$10^5\times J(1.5,-1.5,0.3)$\color{black}}
  ]
\addplot[thick,solid,color=blue,mark=none, mark size=1pt] table [x index=0, y index=1]{fig/noonScloudScattKlight.txt};
\addlegendentry{Total light $J$}
\addplot[thick,solid,color=magenta,mark=none, mark size=1pt] table [x index=0, y index=1]{fig/gemini10absorption.txt};
\addlegendentry{$A_\nu:=1-\e^{-\kappa_\nu}$ }
\addplot[thick,dashed,color=blue,mark=none, mark size=1pt] table [x index=0, y index=1]{fig/gemini11absorption.txt};
\addlegendentry{$A_\nu$+\texttt{CO}$_2$}
\end{axis}
\end{tikzpicture}
\caption{\label{light}  $\kappa_\nu$ from the Gemini measurements and $\kappa_\nu$ modified by adding \texttt{CO}$_2$.  The solid blue curve  is $J(1.5,-1.5,0.5)$versus  wavelength $c/\nu$.}
\end{minipage}
\end{center}
\end{figure}
In a second computation the Gemini values for $\kappa_\nu$ are modified to be 1 in the range  $\nu=(3/18,3/14)$ to simulate an increase of \texttt{CO}$_2$ in the atmosphere. On Figure \ref{Tz}, it is seen that, in this configuration, the \texttt{CO}$_2$ increase the temperature near the ground and decrease it at high altitude. 

 The light intensity $J$ at $(1.5,-1.5,0.5)$ versus wavelength, $c/\nu$, ($c\approx 3$ is the scaled speed of light), is plotted on Figure \ref{light}. Notice that the computation captures complex details due to the discontinuities of $\nu\mapsto\kappa_\nu$.


\section*{Conclusion}
Compressed H-matrices is an ideal tool for RTE in integral form because the complexity of the method is $O(N\sqrt[3]{N}\ln N)$ where $N$ is the number of vertices in the 3D mesh and because it can handle frequency dependent absorption and scattering coefficients at the expense of a finite number of compressed matrices and a finite number of matrix-vector products. 

In the present study, the integral nonlinear formulation of RTE studied in \cite{JCP} has been extended to handle reflective boundary conditions. The monotonicity property of the iterative solver is kept.  The discretization with the finite element method is the same. However it is much harder to write a general computer code because of the complexity of possible multiple reflections, as in ray tracing. Hence in this article the numerical validation has been done only for a finite number of plane reflective boundaries and with at most one reflection per ray.  For the academic test case and for the Chamonix valley  it is essential to add reflective conditions for accuracy.

\subsection*{Acknowledgement}
We would like to thank warmly Fr\'ed\'eric Hecht,  for his constant good will to adapt \texttt{FreeFEM} to our needs. The \texttt{FreeFEM++} app can be downloaded from
{\footnotesize \texttt{www.freefem.org}
}

Some computations were made on the machine \texttt{Joliot-Curie} of the national computing center TGCC-GENCI  under allocation A0120607330.

\bibliographystyle{plain}
\bibliography{references}

\newpage

\section*{Declarations}
\begin{itemize}
\item As to the specific input of each author,  P.-H. Tournier provided the ${\cal H}$-matrices part and the interface between \texttt{htool} and \texttt{FreeFEM}; the rest of the program and the theory were done together.
\item The authors have no relevant financial or non-financial interests to disclose.
\item The authors have no conflicts of interest to declare that are relevant to the content of this article.
\item All authors certify that they have no affiliations with or involvement in any organization or entity with any financial interest or non-financial interest in the subject matter or materials discussed in this manuscript.
\item The authors have no financial or proprietary interests in any material discussed in this article.
\end{itemize}

\newpage

{\Large Entretien avec Jean-Louis Dufresne le 31/3/2023}

\begin{itemize}
\item Compute absorption $A$ or transmission $T$, $A,T\in(0,1)$.
\[
1 - A =T=\e^{\kappa_l\rho l} =\e^{\kappa_m\rho m} 
\]
where $\kappa_m$ has dimension $m^2/kg$ and $\kappa_l\sim m$ called absorption coefficients.

\item  With RT only, $T\mapsto Z$ is dcreasing.  With adiabatic behavior it is a straight line of slope -10K/km (or 6K/km).  See 
Thermal Equilibrium of the Atmosphere with a Given Distribution of Relative Humidity
Syukuro Manabe and Richard T. Wetherald
Journal of the Atmospheric Sciences, Volume 24: Issue 3,Page(s): 241–259
\begin{verbatim}
https://journals.ametsoc.org/view/journals/atsc/
24/3/1520-0469_1967_024_0241_teotaw_2_0_co_2.xml?tab_body=pdf
\end{verbatim}
\item Vents catabatiques
\begin{verbatim}
https://en.wikipedia.org/wiki/Katabatic_wind
\end{verbatim}

\item Clouds : Importance of polarization not clear, (see the LOA lab of Lille).
Scattering due to microparticles:  more important in the direction of the light in front of the particle.

\item Integration of function with respect to $\nu$.  Use Malkmus Narrow-band model. See 

\begin{verbatim}
https://doi.org/10.1364/JOSA.57.000323
https://doi.org/10.1016/S0022-4073(97)00214-8
\end{verbatim}

\end{itemize}
\end{document}